\documentclass{article}
\usepackage{amsmath}
\usepackage{amssymb}
\usepackage{amsfonts}
\usepackage{latexsym,longtable}
\usepackage{tikz}
\usetikzlibrary{matrix}
\begin{document}
\newcommand{\B}{{\cal B}}
\newcommand{\D}{{\cal D}}
\newcommand{\E}{{\cal E}}
\newcommand{\F}{{\cal F}}
\newcommand{\A}{{\cal A}}
\newcommand{\Hh}{{\cal H}}
\newcommand{\Pp}{{\cal P}}
\newcommand{\Z}{{\bf Z}}
\newcommand{\T}{{\cal T}}
\newcommand{\ZZ}{{\mathbb{Z}}}
\newcommand{\qed}{\hphantom{.}\hfill $\Box$\medbreak}
\newcommand{\proof}{\noindent{\bf Proof \ }}
\renewcommand{\theequation}{\thesection.\arabic{equation}}
\newtheorem{theorem}{Theorem}[section]
\newtheorem{lemma}[theorem]{Lemma}
\newtheorem{corollary}[theorem]{Corollary}
\newtheorem{remark}[theorem]{Remark}
\newtheorem{example}[theorem]{Example}
\newtheorem{definition}[theorem]{Definition}
\newtheorem{construction}[theorem]{Construction}


\medskip
\title{New Existence and Nonexistence Results for Strong External Difference Families \thanks{Research supported by NSFC grant 11431003 (L. Ji), 11301457 (Y. Zhang),
 and NSERC RGPIN-16-05610 (R. Wei).
 }}

 \author{{\small   Jingjun Bao$^1$, \ Lijun Ji$^1$, \ Reizhong Wei$^2$\ and Yong Zhang$^3$} \\
 {\small $^1$ Department of Mathematics, Soochow University, Suzhou
 215006, China}\\
 {\small E-mail: jilijun@suda.edu.cn}\\
{\small $^2$ Department of Computer Science, Lakehead University, Thunder Bay, ON P7B 5E1
Canada}\\
 {\small E-mail: rwei@lakeheadu.ca}\\
  {\small $^3$ School of Mathematics and Statistics, Yancheng Teachers University, Jiangsu 224002,  China}\\
 }


\date{}
\maketitle
\begin{abstract}
\noindent \\ In this paper, we use character-theoretic techniques to give new nonexistence results for $(n,m,k,\lambda)$-strong external difference families (SEDFs). We also use cyclotomic classes to give two new classes of SEDFs with $m=2$.

\medskip


\medskip


\end{abstract}


\section{Introduction}

Motivated by applications to algebraic manipulation detection codes (or AMD codes) \cite{CDFPW2008, CFP2013, CPX2015}, Paterson and Stinson introduced strong external difference families (or SEDFs) in \cite{PS2016}. SEDFs are closely related to but stronger than  external difference families (or EDFs) \cite{OKSS2004}. In \cite{PS2016}, it was noted that optimal AMD codes can be obtained from EDFs, whereas optimal strong AMD codes can be obtained from SEDFs. See \cite{PS2016} for a discussion of these and related structures and how they relate to AMD codes.

Recently, Matin and Stinson \cite{MS-arXiv} and Hucznska and Paterson \cite{HP-arXiv} further investigated the existence of
SEDFs and gave some new results about the nonexistence of certain SEDFs. Their results show that the existence of SEDFs
 is an interesting mathematical problem in its own
right, independent of any applications to AMD codes.
In \cite{MS-arXiv}, the authors  presented some nonexistence results of SEDFs by using character theory.
In this paper, we shall explore such a method to give some new nonexistence results for SEDFs.
Let us recall the definition of   SEDFs.

Let $G$ be a finite abelian group of order $n$ (written multiplicatively) with the identity $e\in  G$. For any two nonempty subsets $A_1,A_2$ of $G$, the multiset
$$\Delta_E(A_1,A_2)=\{xy^{-1}\colon x\in A_1, y\in A_2\}$$
is called the external difference of $A_1$ and $A_2$.
Let $k, \lambda, m$ be positive integers and let $A_1,\ldots, A_m$ be (pairwise disjoint) $k$-subsets of $G$. If the following multiset equation holds:
$$\bigcup\limits_{\{\ell \colon \ell\neq j\}}\Delta_E(A_{\ell},A_j)=\lambda (G\setminus \{e\})$$
for each $1\leq j\leq m$,
then the collection $\{A_1,\ldots, A_m\}$ is denoted as an $(n,m,k,\lambda)$-SEDF.

When the multiset equation is replaced with $$\bigcup\limits_{\{\ell,j \colon \ell\neq j\}}\Delta_E(A_{\ell},A_j)=\lambda (G\setminus \{e\}),$$ the collection $\{A_1,\ldots, A_m\}$ is denoted as an $(n, m, k, \lambda)$-EDF. Clearly, an $(n,m,k,\lambda)$-SEDF is an $(n,m,k,m\lambda)$-EDF.

From the definition of an SEDF, it is easy to see that $m>1$, $mk\leq n$ and
\begin{equation}
\label{mk2=lambda(n-1)}
(m-1)k^2=\lambda (n-1).
\end{equation} Let us recall some known results for SEDFs.

\begin{lemma}[\cite{HP-arXiv}]
For an $(n,m,k,\lambda)$-SEDF, either

{\rm (1)} $k=1$ and $\lambda=1$; or

{\rm (2)} $k>1$ and $\lambda <k$.

\end{lemma}

\begin{theorem}
[\cite{PS2016}]
There exists an $(n, m, k, 1)$-SEDF if and only if $m = 2$ and $n = k^2 + 1$,
or $k = 1$ and $m = n$.
\end{theorem}

Recent paper \cite{MS-arXiv}, using character theory, has established various 
non-existence results, including the following:

\begin{theorem}[\cite{MS-arXiv}]
\label{m34}
 If there is an $(n,m,k,\lambda)$-SEDF with $k>1$, then $m\neq 3$ and $m\neq 4$.

\end{theorem}

\begin{theorem}[\cite{MS-arXiv}]\label{t.stinson}
If $G$ is any group of prime order, $k > 1$ and $m > 2$, then there does not exist  an $(n, m, k, \lambda)$-SEDF over $G$.
\end{theorem}

Recent paper \cite{HP-arXiv} gave some necessary conditions for the existence of SEDFs with $\lambda \geq 2$ as follows.

\begin{theorem}[\cite{HP-arXiv}]
\label{NC2}
Suppose there exists an $(n, m, k, \lambda)$-SEDF with $m\geq 3$ and $k> \lambda \geq 2$. Then the
following inequality must hold:
$$\frac{\lambda(k-1)(m-2)}{(\lambda-1)k(m-1)}\leq 1.$$
\end{theorem}

In \cite{HP-arXiv}, Huczynska and Paterson pointed out that an $(n, m, k, 2)$-SEDF can exist only when $m = 2$ and gave an infinite classes of SEDFs.

\begin{theorem}
[\cite{HP-arXiv}]
For any prime power $q$ with $q\equiv  1 \pmod 4$, there exists a $(q, 2, \frac{q-1}{2},\frac{q-1}{4})$-SEDF over the finite field $GF(q)$.
\end{theorem}

Martin and Stinson showed that there does not exist an $(n, m,k,\lambda)$-SEDF with $n=mk$ and $k>1$ \cite{MS-arXiv}. This can be improved as follows.

\begin{lemma}
\label{NC1}
There does not exist an $(n, m,k,\lambda)$-SEDF with  $k>1$ and $n$ being divisible by $k$.
\end{lemma}

\proof From (\ref{mk2=lambda(n-1)}), we have $(m-1)k^2=\lambda(n-1)$. As $n$ is divisible by $k$, we have gcd$(k,n-1)=1$, thereby $\lambda$ is divisible by $k^2$. Let $\lambda =tk^2$ for some positive integer $t$, and hence $m-1=t(n-1)$. This shows that $t=1$ and $n=m$. As an SEDF contains $m$ pairwise disjoint subsets, we have $n\geq mk$. This shows that $k=1$. \qed

In this paper, we also use character theory to give some new necessary conditions for the existence of SEDFs and present new families of SEDFs by using cyclotomic classes.

The rest of this paper is organized as follows. In Section 2, we  review some basic facts about characters of finite abelian groups. In Section 3, we use character theory to present some new necessary conditions for the existence of SEDFs and to give some non-existence results for SEDFs. In Section 4, cyclotomic classes are used to present two new classes of SEDFs with $m=2$. Finally, Section 5 concludes this paper.


\section{Preliminaries}

We briefly review some basic facts about characters of finite abelian groups. These can be found, for example,
in \cite{Ledermann1987}.

For a finite abelian group $G$, there are exactly $n = |G|$ distinct homomorphisms (called {\em characters}) from $G$ to the multiplicative
group of complex numbers. In particular, the character $\chi_0\colon G\rightarrow \mathbb{C}$ defined by $\chi_0(g)=1$ for all $g\in G$ is called the {\em principal character}.
When $G$ is abelian, the product of characters $(\chi\psi)(g) =\chi(g)\psi(g)$ is again a character and, provided $G$ is finite, this gives us a group $\widehat{G}$ of characters isomorphic to $G$. So we can label the $|G|$ distinct
characters $\{\chi_g\colon g\in G\}$.

These characters $\chi, \psi, \ldots$ satisfy the following ``orthogonality" relations: 
\[
\sum\limits_{g\in G}\chi(g)\overline{\psi(g)}
=\left \{ \begin{array}{ll}
n & {\rm if}\ \chi=\psi,\smallskip \\
0 & {\rm otherwise}.
\end{array}
\right .
\] and
\[
\sum\limits_{\chi\in \widehat{G}}\chi(g)\overline{\chi(h)}
=\left \{ \begin{array}{ll}
n & {\rm if}\ g=h,\smallskip \\
0 & {\rm otherwise}.
\end{array}
\right .
\]
So, for any non-principal character, we have
$\sum_{g\in G}
\chi(g) = 0$. Likewise
$\sum_{g\in G}\chi_g(h)= 0$ unless $h$ is the identity $e$
of $G$, in which case the sum equals $n$.

Let $\mathbb{Z}$ denote the ring of integers. The group ring $\mathbb{Z}[G]$ is defined to be the ring of formal sums
$$\mathbb{Z}[G]=\left \{ \sum\limits_{g\in G} a_gg\colon a_g\in \mathbb{Z}\right \},$$ where the addition is given by
$$\sum\limits_{g\in G}a_gg+\sum\limits_{g\in G}b_gg=\sum\limits_{g\in G}(a_g+b_g)g,$$
and the multiplication is defined by
$$\left (\sum\limits_{g\in G}a_gg\right )\left (\sum\limits_{g\in G}b_gg\right )=\sum\limits_{h\in G}\left (\sum\limits_{g\in G}a_gb_{hg^{-1}} \right )h.$$
If $S$ is a multiset of $G$ where for $g\in G$, $g$ occurs in $S$ exactly $a_g$ times, we will identify $S$ with the group ring element $S=\sum_{g\in G}a_gg$.
Also, $S^{-1}=\sum_{g\in G}a_gg^{-1}$.

Each $\chi$ extends to an algebra homomorphism from the group algebra $\mathbb{Z}[G]$ to $\mathbb{C}$ by
$$\chi\left(\sum_{g\in G}a_gg\right )=\sum_{g\in G}a_g\chi(g).$$

Let $A_1,\ldots, A_m$ be an $(n,m,k,\lambda)$-SEDF over $G$ and let $H$ be a subgroup of $G$. Let $\sigma$ be a natural homomorphisms from $G$ to the quotient group $G/H$.
Define $D_j=\sigma(A_j)$. Then $D_j\in \mathbb{Z}[G/H]$ and $\sum\limits_{\ell\neq j}D_jD_{\ell}^{-1}=\lambda |H| (G/H)-\lambda (H)$ for $1\leq j\leq m$. We shall use character theory to find some properties on $D_1, \ldots, D_m$.


\section{New nonexistence results for SEDFs}

In this section, we use character theory to give some new nonexistence results for SEDFs. 
Our proofs of Lemmas 3.1-3.3 are very similar to the proofs in [8]. The only difference is that we prove the results in the quotient group. We find that their proofs are not only suitable for the SEDFs, but also work for the images of the SEDFs in the quotient group. However, we still include these proofs for the completeness and for the convenience of readers.

Let $G$ be an abelian group of order $n$ with the identity $e$ and let  $D_j=\sum_{g\in G}a_{j,g}g\in \mathbb{Z}[G]$, $1\leq j\leq m$,
 satisfy
\begin{equation}
\label{EDE}
\left \{
\begin{array}{l}
a_{j,g}\geq 0\ {\rm for}\ g\in G, \smallskip \\
\sum_{g\in G}a_{j,g}=k, \smallskip \\
\sum\limits_{\ell\neq j}D_jD_{\ell}^{-1}=\lambda G-\mu e\ {\rm for}\ 1\leq j\leq m,
\end{array}
\right .
\end{equation}
where $\lambda, \mu$ are positive integers and $\lambda \geq \mu$.

\begin{lemma}
\label{ExistenceNonPrincipal}
Let $D\in \mathbb{Z}[G]$. Except when  $D= tG$ for some nonnegative integer $t$,  there is at least one non-principal character $\chi$ of $G$ satisfying ${\chi(D)}\neq 0$.
\end{lemma}

\proof Let $G=\{g_1,\ldots g_n\}$ where $g_1=e$ and $D=a_1g_1+\cdots +a_ng_n$.
Let $F$ be the ``Fourier matrix" whose rows are indexed by the elements of $G$, whose columns are indexed by the characters of $G$ and whose $(g,\chi)$-element is $\chi (g)$. Without loss of generality, let the first row is indexed by the identity of $G$ and the first row is indexed by the principal character. It is well known from the orthogonal relations that $F$ is invertible and $F\overline{F}^T=nI_n$ where $\overline{F}^T$ means the transposition of $\overline{F}$ and $I_n$ denotes the identity matrix of order $n$.

Suppose that $D\neq \emptyset$ and $\chi(D)=0$ for any non-principal character $\chi$ of $G$. Then $$(\chi_{g_1}(D), \chi_{g_2}(D),\ldots \chi_{g_n}(D))=\left(\sum_{1\leq j\leq n}a_{j},0,\ldots,0\right).$$
Since $(\chi_{g_1}(D), \chi_{g_2}(D),\ldots \chi_{g_n}(D))=(a_1,a_2,\ldots, a_n)F$, we have
\[
\begin{array}{ll}
(a_1,a_2,\ldots, a_n)&=(\sum\limits_{1\leq j\leq n}a_{j},0,\ldots,0)F^{-1} \smallskip \\
&=(\sum_{1\leq j\leq n}a_{j},0,\ldots,0)\cdot \frac{1}{n} \overline{F}^T \smallskip \\
&=\left (\frac{\sum\limits_{1\leq j\leq n}a_{j}}{n},\ldots,\frac{\sum\limits_{1\leq j\leq n}a_{j}}{n} \right ).
\end{array}
\]
It follows that $D=\frac{\sum_{1\leq j\leq n}a_{j}}{n}G$. \qed

\begin{lemma}
\label{chi(Dj)}
Let  $D_1,\ldots,D_m\in \mathbb{Z}[G]$ satisfy {\rm (\ref{EDE})} and $\D=\sum_{1\leq j\leq m}D_j$. Suppose that a non-principal character $\chi$ of $G$ satisfies $\chi(\D)\neq 0$. Then there exist nonzero real numbers $\alpha_1,\ldots,\alpha_m$ such that $\chi(D_j)=\alpha_j\chi(\D)$ for $j=1,\ldots,m$. Further, for $m>3$, there can be only two possible  values for $\alpha_j$, $1\leq j\leq m$.
\end{lemma}

\proof Since $\sum_{\ell\neq j}D_jD_{\ell}^{-1}=D_j\D^{-1}-D_jD_j^{-1}=\lambda G-\mu e$, we have
\begin{equation}
\label{chi(Dj)Equation} \sum_{\ell\neq j}\chi(D_j)\overline{\chi(D_{\ell})}=\chi(\lambda G-\mu e)=-\mu.
\end{equation}
 It follows that $\chi(D_j)\neq 0$ and
$\chi(D_j)\overline{\chi(\D)}-\chi(D_j)\overline{\chi(D_j)}=-\mu,$
i.e., $\overline{\chi(D_j)}\chi(\D)=\chi(D_j)\overline{\chi(D_j)}-\mu.$ Since $\chi(\D)\neq 0$ by assumption, we have $\chi(D_j)\overline{\chi(D_j)}-\mu\neq 0$. Multiplying both sides by $\chi(D_j)$, we obtain $\chi(D_j)=\frac{\chi(D_j)\overline{\chi(D_j)}}{\chi(D_j)\overline{\chi(D_j)}-\mu}\chi(\D)$, so,
$$\alpha_j=\frac{\chi(D_j)\overline{\chi(D_j)}}{\chi(D_j)\overline{\chi(D_j)}-\mu}$$
is nonzero. Further, $\alpha_j$ is clearly a real number.

For $m>3$, let $r,s$ be two distinct indices. By (\ref{chi(Dj)Equation}) we find
$$\alpha_r\chi(\D)\cdot \alpha_s\overline{\chi(\D)}+\sum\limits_{\ell\neq r,s}\alpha_r\chi(\D)\alpha_{\ell}\overline{\chi(\D)}=-\mu=\alpha_s\chi(\D)\cdot \alpha_r\overline{\chi(\D)}+\sum\limits_{\ell\neq r,s}\alpha_s\chi(\D)\alpha_{\ell}\overline{\chi(\D)}.$$
Recall we are assuming $\chi(\D)\neq 0$. So we obtain
\begin{equation}
\label{alpha1}
\alpha_r\left (\sum_{\ell\neq r,s}\alpha_{\ell}\right )=\alpha_s\left (\sum_{\ell\neq r,s}\alpha_{\ell}\right ).
\end{equation}
 If $\alpha_r\neq \alpha_s$, then $\sum\limits_{\ell\neq r,s}\alpha_{\ell}=0$ and $\alpha_t=-\sum\limits_{\ell\neq r,s,t}\alpha_{\ell}$. Likewise, if $\alpha_r\neq \alpha_t$ then $\alpha_s=-\sum\limits_{\ell\neq r,s,t}\alpha_{\ell}$, i.e., $\alpha_s=\alpha_t$. \qed

We assume that $m>3$, and $$\alpha_1,\ldots \alpha_m\in \{\alpha,\beta\}$$
with $\alpha_j=\alpha$ for exactly $a$ values of $j$ and  with $\alpha_j=\beta$ for exactly $b$ values of $j$. Without loss of generality, we assume $0\leq a\leq b$, and we known that $a+b=m$. The unordered pair $(a,b)$ will be called the spit of $(\alpha_1,\ldots,\alpha_m)$.

\begin{lemma}
\label{HighlyUnevern}
Let  $D_1,\ldots,D_m\in \mathbb{Z}[G]$ satisfy {\rm (\ref{EDE})}, $m>3$ and $\D=\sum_{1\leq j\leq m}D_j$. Suppose that a non-principal character $\chi$ of $G$ satisfies $\chi(\D)\neq 0$. Then the split of $(\alpha_1,\ldots,\alpha_m)$ cannot be $(0,m)$, $(1,m-1)$ or $(m/2,m/2)$.
\end{lemma}

\proof Since $\chi(\D)=\sum_{1\leq j\leq m}\chi(D_j)=\sum_{1\leq j\leq m}\alpha_j\chi(\D)$, we have
\begin{equation}
\label{a alpha+ b beta=1}
 a\alpha+b \beta=1.
\end{equation}
By (\ref{chi(Dj)Equation}) we have
\begin{equation}
\label{alpha}
\alpha_j\left ( \sum\limits_{\ell\neq j}\alpha_{\ell}\right )=-\frac{\mu}{\chi(\D)\overline{\chi(\D)}}.
\end{equation}

If $a=0$ then $\beta=1/m$ by (\ref{a alpha+ b beta=1}). But this gives a positive value for the left-hand side of (\ref{alpha}) while the right-hand side is negative, a contradiction.

Next, if $a=1$ and $b=m-1$, we employ (\ref{alpha1}) to find $\alpha((m-2)\beta)=\beta((m-2)\beta)$. Since $\beta\neq 0$ by Lemma \ref{chi(Dj)}, this gives $\alpha=\beta$, which is the case we just discussed.

Finally, if $m$ is even and $a=b$, then  (\ref{alpha1})  with $\alpha_r=\alpha$ and $\alpha_s=\beta$, $\alpha \neq \beta$,
 gives $\alpha+\beta=0$. But this is impossible because, if $a=b=m/2$, we must have $\alpha+\beta=2/m$
  from (\ref{a alpha+ b beta=1}). \qed

\begin{lemma}
\label{NC2}
Let $p$ be a prime and let $G$ be a group of order $p$.  Let $D_1,\ldots,D_m\in \mathbb{Z}[G]$ satisfy {\rm (\ref{EDE})}. If $k>1$, $m>3$, $\lambda$ is divisible by  $\mu$ and $\sum_{1\leq j\leq m}D_j\neq \theta G$ for any positive integer $\theta$, then $(m-2)k\equiv 0\pmod p$.
\end{lemma}

\proof  Since $p$ is a prime, $G$ is a cyclic group. Let $G$ be generated by $z$. 
Denote $\D=\sum_{1\leq j\leq m}D_j$.  By assumption and Lemma \ref{ExistenceNonPrincipal}, there is a non-principal character $\chi$ of $G$ satisfying $\chi(\D)\neq 0$.  
Let $\xi=\chi(z)$. Then $\xi$ is a primitive $p${th} root of unity.
By  Lemma \ref{chi(Dj)}, we have two distinct real numbers $\alpha, \beta$ such that $\alpha_j\in \{\alpha, \beta\}$ for $1\leq j\leq m$, with each value occurring at least twice. Choose $r$ and $s$ with $\alpha_r=\alpha$ and $\alpha_s=\beta$.  Then Equation (\ref{alpha1}) forces
\begin{equation}
\label{D}
\chi\left (\sum\limits_{\ell\neq r,s}D_{\ell}\right )=0.
\end{equation}
Clearly, the minimal polynomial of $\xi$ over the rational filed $Q$ is the $p$th  cyclotomic polynomial $\Phi(x)=x^{p-1}+x^{p-2}+\cdots +x+1$.
Since the left-hand side of (\ref{D}) is a polynomial of $\xi$, the equation (\ref{D}) forces
$$\sum\limits_{\ell\neq r,s}D_{\ell}=\theta' G$$
for some integer $\theta'$.  Applying the principal character $\chi_0$ of $G$ on the above equation yields $(m-2)k=\theta' p$. \qed

\begin{theorem}
\label{NC}
If there is an $(n,m,k,\lambda)$-SEDF over an abelian group $G$ with $k>1$ and $m>4$, then the following hold:

\noindent $(1)$ $n\not\equiv 0 \pmod k$;

\noindent $(2)$ for any prime divisor $p$ of $n$, $m\equiv 2\pmod p$ whenever gcd$(mk,p)=1$.
\end{theorem}

\proof The first claim holds by Lemma \ref{NC1}.

Suppose that $\{A_1,\ldots, A_m\}$ is an $(n,m,k,\lambda)$-SEDF over $G$. It is well known that there is a subgroup $H$ of order $\frac{n}{p}$. Let $\sigma$ be a natural homomorphisms from $G$ to the quotient group $G/H$.
Define $D_j=\sigma(A_j)$ for $1\leq j\leq m$. Then $D_j\in \mathbb{Z}[G/H]$ and $\sum\limits_{\ell\neq j}D_jD_{\ell}^{-1}=\lambda |H| (G/H)-\lambda (H)$, i.e., $\{D_1,\ldots,D_m\}$ satisfies (\ref{EDE}). Note that $G/H$ is an abelian group of order $p$. When gcd$(mk,p)=1$, it holds that $\sum_{1\leq j\leq m}D_j\neq \theta (G/H)$ for any positive integer $\theta$. By Lemma \ref{NC2} we obtain the second claim. \qed

\noindent{\bf Remark:}  When $n= p$ is a prime, we must have
$\gcd (mk, p) =1$. So, Theorem \ref{t.stinson} is obtained again by applying Theorem \ref{NC}.

\begin{corollary}
\label{p1ps}
If $n = p_1p_2 ... p_s$ where $p_1,\ldots,p_s$ are distinct  primes, then there is no $(n,m,k,\lambda)$-SEDFs for $m > 4$ and
$k > 1$ when gcd $(mk, n) = 1$.
\end{corollary}

\proof Since gcd$(mk,n)=1$, by Theorem \ref{NC} we would have $m\equiv 2\pmod{p_i}$ for $1\leq i\leq s$ if there were an $(n,m,k,\lambda)$-SEDF. Since $p_1,\ldots,p_s$ are distinct  primes, we have $m\equiv 2\pmod n$, a contradiction. \qed

\begin{theorem}
Let $p$ be a prime, $k>1$ and $m>2$. Then there does not exist a $(p^2,m,k,\lambda)$-SEDF over a cyclic group of order $p^2$.
\end{theorem}

\begin{proof}
By Theorem \ref{m34}, need only consider $m\geq 5$. We first show that there does not exist a $(p^2,m,k,\lambda)$-SEDF over a cyclic group of order $p^2$ if $k\equiv 0\pmod p$. 
If $k=xp$ then by the fact $(m-1)k^2=\lambda(p^2-1)$ we have that $\lambda$ is divisible by $p^2$, thereby $mx<p$ and $(m-1)x^2\geq p^2-1$. It is impossible because $(m-1)x^2<(m-1)\frac{p^2}{m^2}<p^2-1$. Now we assume $k\not\equiv 0\pmod p$. 

 Let $G$ be a cyclic group of order $p^2$ generated by $z$. 
Suppose, on the contrary, that there is a $(p^2,m,k,\lambda)$-SEDF
$\{A_1,\ldots,A_m\}$ over $G$.
Let $\chi$ be the character of $G$ defined by $\chi(z^k)=exp^{\frac{2k\pi i}{p^2}}$ for each integer $k$ where $exp^{\frac{2\pi i}{p^2}}$ is a $p^2$th primitive root of unity.

If $\chi(\sum_{1\leq j\leq m}A_j)=0$, then by \cite[Theorem 3.3]{LL2000} $\sum_{1\leq j\leq m}A_j$ can be represented as $(1+z^{p}+z^{2p}+\cdots+z^{p^2-p})g(z)$ for some $g(z)\in \mathbb{Z}[G]$ with all coefficients of $g(z)$ being non-negative. Since the coefficients of   $\sum_{1\leq j\leq m}A_j$ are 0's or 1's and the degree of $\sum_{1\leq j\leq m}A_j$ is less than $p^2$, we have the degree of $g(z)$ is less than $p$ and all coefficients of $g(z)$ are 0's or $1$'s. Therefore, $\sum_{1\leq j\leq m}A_j$ is expressible as a disjoint union of cosets of $\langle z^p\rangle$. When $\sum_{1\leq j\leq m}A_j=G$, we have $p^2\equiv 0\pmod k$, which is impossible by Theorem \ref{NC}.  When $\sum_{1\leq j\leq m}A_j\neq G$, the set of multisets $\{\sigma(A_1),\ldots, \sigma(A_m)\}$ satisfies Equations (\ref{EDE}) and $\sigma(\sum_{1\leq j\leq m}A_j)\neq \theta G/\langle z^p\rangle$ for any integer $\theta$,
where $\sigma$ is a natural homomorphisms from $G$ to the quotient group $G/\langle z^p\rangle$. By Lemma \ref{NC2}, we have $(m-2)k\equiv 0\pmod p$.  Since $mk\equiv 0\pmod p$, we have $m-2\equiv m\equiv 0\pmod p$ or $k\equiv 0\pmod p$. Then $p=2$ or $k\equiv 0\pmod p$, a contradiction.

If $\chi(\sum_{1\leq j\leq m}A_j)\neq 0$, then By Lemma \ref{HighlyUnevern} we  have two distinct real numbers $\alpha,\beta$ such that $\alpha_j\in \{\alpha,\beta\}$ for $1\leq j\leq m$, with each value occurring at least twice. Without loss of generality, let $\alpha_1=\cdots =\alpha_a=\alpha$ and $\alpha_{a+1}=\cdots =\alpha_m=\beta$.  Then Equation (\ref{alpha1}) forces
$$\chi \left (\sum\limits_{\ell\neq 1,m}A_{\ell}\right )=0.$$
Similar discussion as above shows that
$\sum\limits_{\ell\neq 1,m}A_{\ell}$ is expressible as a disjoint union of cosets of $\langle z^p\rangle$ of $G$.
Similarly, $\sum\limits_{\ell\neq 1,m-1}A_{\ell}$ is also expressible as a disjoint union of cosets of $\langle z^p\rangle$ of $G$. If $A_m$ were not a disjoint union of cosets of the $\langle z^p\rangle$ of $G$, then there would be a nonempty  subset $T \subset A_m$ such that $T$ is a proper subset of a coset $z^t\langle z^p\rangle$  and $z^t\langle z^p\rangle\setminus T$ is a subset of  $\sum\limits_{\ell\neq 1,m-1,m}A_{\ell}$. Since $A_{m-1}\cap A_m=\emptyset$, $\sum\limits_{\ell\neq 1,m}A_{\ell}$ does not contain $z^t\langle z^p\rangle$, i.e., not a
disjoint union of cosets of $\langle z^p\rangle$ of $G$,  a contradiction. Therefore, each $A_j$ is a disjoint union of cosets of the subgroup $\langle z^p\rangle$ of $G$, while this leads to $\sum_{\ell \neq j}A_jA_{\ell}^{-1}$ having  no elements of $\langle z^p\rangle$, a contradiction.  \qed
\end{proof}

For $v=pq$ where $p,q$ are distinct primes, Martin and Stinson analyzed the properties of a $(pq,m,k,\lambda)$-SEDF and failed to rule out its existence. In order to give  its nonexistence, we need the following result which, as pointed out in \cite{MS-arXiv}, can be obtained from \cite{LL2000}.

\begin{theorem}
[\cite{LL2000}]
\label{pq}
Let $p,q$ be distinct primes, $G$ a cyclic group of order $pq$ and $\varepsilon\subset G$. If there is a non-principal character $\chi$ of $G$ satisfying  $\chi(\varepsilon)=0$, then there is a nonnegative integer $k$ such that $|\varepsilon|=kp$ or $|\varepsilon|=kq$. Further, $\varepsilon$ is expressible as a disjoint union of cosets of some proper subgroup $H$ of $G$ (where $|H|=p$ or $|H|=q$).
\end{theorem}

\begin{theorem}
Let $p,q$  be distinct primes, $k>1$ and $m>2$. Then there  does not exist a $(pq,m,k,\lambda)$-SEDF.
\end{theorem}

\proof By Theorem \ref{m34}, need only consider $m\geq 5$. Since $p,q$  are distinct primes, a group $G$ of order $pq$ is a cyclic group. Suppose, on the contrary, that $\{A_1,\ldots, A_m\}$ is a $(pq,m,k,\lambda)$-SEDF over $G$.
Then $A=\sum_{1\leq j\leq m}A_j\neq G$ by Theorem \ref{NC}.  By Lemma \ref{ExistenceNonPrincipal}, there would be a non-principal character $\chi$ of $G$ satisfying $\chi(A)\neq 0$. By Lemma \ref{HighlyUnevern}, we would have two distinct real numbers $\alpha,\beta$ such that $\alpha_j\in \{\alpha,\beta\}$ for $1\leq j\leq m$, with each value occurring at least twice. Without loss of generality, let $\alpha_1=\cdots =\alpha_a=\alpha$ and $\alpha_{a+1}=\cdots =\alpha_m=\beta$.  Then Equation (\ref{alpha1}) forces
$$\chi \left (\sum\limits_{\ell\neq 1,m}A_{\ell}\right )=0.$$
By Theorem \ref{pq}, $\sum\limits_{\ell\neq 1,m}A_{\ell}$ would be a disjoint union of cosets of some proper subgroup $H$ of $G$ (where $|H|=p$ or $|H|=q$).
Similarly, $\sum\limits_{\ell\neq 1,m-1}A_{\ell}$ would also be a disjoint union of cosets of the subgroup $H$ of $G$. If $A_m$ were not a disjoint union of cosets of the subgroup $H$ of $G$, then there would be a nonempty  subset $T \subset A_m$ such that $T$ is a proper subset of a coset $H'$ of $H$ of $G$ and $H'\setminus T$ is a subset of  $\sum\limits_{\ell\neq 1,m-1,m}A_{\ell}$. Since $A_{m-1}\cap A_m=\emptyset$, $\sum\limits_{\ell\neq 1,m}A_{\ell}$ does not contain $H'$, i.e., not a
disjoint union of cosets of the subgroup $H$ of $G$,  a contradiction. Therefore, each $A_j$ is a disjoint union of cosets of the subgroup $H$ of $G$, while this leads to $\sum_{\ell \neq j}A_jA_{\ell}^{-1}$ having  no elements of $H$, a contradiction. \qed

\section{ A construction of SEDFs via cyclotomic classeses}

In this section, we use cyclotomic classes of index 4 and 6  to present two classes of SEDFs.

Let $q$ be a prime power with $q=ef+1$ and $GF(q)$ be the finite field of $q$ elements. Given a primitive element $\alpha$ of $GF(q)$, define $C_0^{e}=\{\alpha^{je}\colon 0\leq j\leq f-1\}$, the multiplicative group generated by $\alpha^e$, and $$C_i^{e}=\alpha^iC_0^{e}$$
for $1\leq i\leq e-1$. Then $C_0^{e},  C_1^{e},\ldots,  C_{e-1}^{e}$ partition $GF(q)^*=GF(q)\setminus \{0\}$. The $C_i^{e}$ $(0\leq i<e)$ are known as {\em cyclotomic classes} of index $e$ (with respect to $GF(q)$). 

In the theory of cyclotomy, the numbers of solutions of
$$x+1=y,\ x\in C_i^{e},\ y\in C_j^{e}$$
are called {\em cyclotomic numbers} of order $e$ respect to $GF(q)$ and denoted by $(i,j)_e$.

\begin{construction}
\label{SEDF-Cyclotomic}
Let $q = ef + 1 \equiv  1 \pmod 4$ be a prime power, where $e$ is even. Then
$C^e_0,C^e_{\frac{e}{2}}$
forms a $(q,2,\frac{q-1}{e},\frac{q-1}{e^2})$-SEDF if all $(i,\frac{e}{2})_e$, $0\leq i<e$, are equal.
\end{construction}

\proof Set $G=GF(q)$, $A_1=C_0^{e}$ and $A_2=C_{e/2}^{e}$. Computing the external differences of $A_1$ and $A_2$, we have
$$\Delta_E(A_{2},A_1)=\Delta_E(C_{e/2}^{e},C_0^{e})=\{x-1:~x\in C_{e/2}^e\}\cdot C_0^{e}.$$
Since $|\{x-1:~x\in C_{e/2}^e\}\cap C_i^e|=|C_{e/2}^e\cap (C_i^e+1)|=(i,\frac{e}{2})_e$ for $0\leq i<e$, we have
$$\Delta_E(C_{e/2}^{e},C_0^{e})=\{x-1:~x\in C_{e/2}^e\}\cdot C_0^{e}=\bigcup\limits_{i=0}^e(i,\frac{e}{2})_eC_i^e.$$
Clearly, $\Delta_E(A_{2},A_1)=-\Delta_E(A_{1},A_2)$.
Then, $\{C_0^{e},C_{e/2}^{e}\}$ is a strong external difference family if and only if $\lambda=(0,\frac{e}{2})_e=(1,\frac{e}{2})_e=\cdots =(e-1,\frac{e}{2})_e$. \qed

\begin{lemma}[\cite{Storer1967}]
\label{4 cyclotomic numbers}
Let $q=4f+1$ be a prime power, where $f$ is even. The cyclotomic numbers $(0,2)_4,(1,2)_4, (2,2)_4, (3,2)_4,$ are determined by the relations
$$ 16(0,2)_4=16(2,2)_4=q-3+2s,$$
$$ 16(1,2)_4= 16(3,2)_4=q+1-2s,$$
\noindent
where $q=s^{2}+4t^{2}, s\equiv 1\pmod 4$ is the proper representation of $q$ if $ q\equiv 1\pmod 4$; the sign of $t$ is ambiguously determined.
\end{lemma}

\begin{theorem}
Let $q$ be a prime power such that $q=1+16t^2$, then there exists a $(q, 2, \frac{q-1}{4},\frac{q-1}{16})$-SEDF.
\end{theorem}

\proof Since $q=1+16t^2$, i.e., $s=1$, by Lemma \ref{4 cyclotomic numbers} we have
$$16(0,2)_4=16(2,2)_4=q-3+2s=q+1-2s=16(1,2)_4=16(1,3)_4.$$
So, $\{C_0^{4},C_{2}^{4}\}$ is a  $(q, 2, \frac{q-1}{4},\frac{q-1}{16})$-SEDF by Construction \ref{SEDF-Cyclotomic}. \qed

\begin{lemma}  [\cite{Whiteman1960}]
\label{6 cyclotomic numbers}
Let $p=6f+1$ be a prime, where $f$ is even. When $2\in C_0^3$, the cyclotomic numbers $(0,3)_6,(1,3)_6,(2,3)_6,(3,3)_6,(4,3)_6,(5,3)_6$, are determined by the relations
$$ 36(0,3)_6=36(3,3)_6=q-5+4s,$$
$$ 36(1,3)_6=36(4,3)_6=q+1-2s,$$
$$ 36(2,3)_6=36(5,3)_6=q+1-2s,$$
\noindent
where $p=s^{2}+3t^{2}, s\equiv 1\pmod 6$ is the proper representation of $p$; the sign of $t$ is ambiguously determined.
\end{lemma}

\begin{corollary}[\cite{IR1990}]
\label{3 character}
Let $p\equiv 1\pmod 6$ be a prime. Then $2\in C_0^3$ in $GF(p)$ if and only if $p=x^2+27y^2$.
\end{corollary}

\begin{theorem}
Let $p$ be a prime  such that $p=1+108t^2$, then there exists a $(p, 2, \frac{p-1}{6},\frac{p-1}{36})$-SEDF.
\end{theorem}

\proof From Corollary \ref{3 character}, we have $2\in C_0^3$ for $p=1+108t^2$. Then, by Lemma \ref{6 cyclotomic numbers} we have
$$36(0,3)_6=p-5+4s=36(1,3)_6=p+1-2s=36(2,3)_6=p+1-2s,$$ where $s=1$.
So, $\{C_0^{6},C_{3}^{6}\}$ is a  $(p, 2, \frac{p-1}{6},\frac{p-1}{36})$-SEDF by Construction \ref{SEDF-Cyclotomic}. \qed

\section{Conclusion}

In this paper, we use character-theoretic techniques to prove some new necessary conditions for the
existence of SEDFs, which gives new nonexistence results for certain types of SEDFs. The results generalized
the results from \cite{MS-arXiv}. We also use cyclotomoic classes to construct two new infinite classes of
$(n, 2, k, \lambda )$-SEDFs. All known results for SEDFs with $m=2$ are  using groups   of prime power order.
Therefore an interesting open problem is: Can we construct an $(n, 2, k, \lambda )$-SEDF with $n$ not a prime
power for $k, \lambda > 1$?

From the known necessary conditions and  constructions, we also make the following conjecture.

\noindent{\bf Conjecture:} There does not exist an $(n, m, k, \lambda )$-SEDF for $m > 2, k > 1$ and $\lambda > 1$.




\end{document}